\documentclass[11pt]{article}

\usepackage[margin=1in]{geometry}
\usepackage[T1]{fontenc}
\usepackage{lmodern}
\usepackage{amsmath,amssymb,amsthm}
\usepackage{microtype}
\usepackage{needspace}
\usepackage[numbers,sort&compress]{natbib}
\usepackage{xurl}
\usepackage[hidelinks]{hyperref}

\newtheorem{theorem}{Theorem}
\newtheorem{lemma}[theorem]{Lemma}
\newtheorem{corollary}[theorem]{Corollary}
\theoremstyle{remark}
\newtheorem{remark}[theorem]{Remark}
\DeclareMathOperator{\sr}{sr}
\DeclareMathOperator{\pr}{pr}
\DeclareMathOperator{\rank}{rank}

\newcommand{\F}{\mathbb F}
\newcommand{\cS}{\mathcal S}
\newcommand{\xx}{\mathbf x}
\newcommand{\doi}[1]{\href{https://doi.org/#1}{doi:\nolinkurl{#1}}}
\hypersetup{pdftitle={Limitations of the slice rank method in additive combinatorics},
  pdfauthor={Sankeerth Rao Karingula and Shachar Lovett}}

\title{Limitations of the slice rank method in additive combinatorics}
\author{
Sankeerth Rao Karingula\\
{Agentin AI}\\
{\texttt{sankeerthrao@agentin.ai}}
\and
Shachar Lovett\thanks{Supported by Simons Investigator Award \#929894 and NSF award CCF-2425349.}\\
{University of California, San Diego}\\
{\texttt{slovett@ucsd.edu}}
}
\date{September 2026}

\begin{document}
\maketitle

\begin{abstract}
The slice rank method gives exponential bounds for sets with no
three-term arithmetic progression in finite vector spaces of odd
characteristic and for three-sunflower-free families of subsets of a
fixed ground set. We show that for $k\ge4$, every tensor that is
nonzero exactly on the $k$-term arithmetic progression relation or
the $k$-sunflower relation has maximal slice rank over every coefficient
field. When the support is prescribed only on pairwise distinct inputs,
we obtain comparable lower bounds, which likewise rule out exponential savings.
\end{abstract}

\section{Introduction}

The polynomial method of Croot, Lev, and Pach~\cite{CLP17} gave an
exponential bound for progression-free subsets of $(\mathbb Z/4\mathbb Z)^n$.
Ellenberg and Gijswijt~\cite{EG17} adapted it to three-term progressions
in $\mathbb F_q^n$, for odd $q$, resolving the cap-set problem when $q=3$.
Tao~\cite{Tao16} gave the symmetric tensor formulation that introduced
slice rank. The method separates a combinatorial argument into two
parts: construct a tensor of small slice rank, and show that its
restriction to a configuration-free set is a nonzero diagonal tensor.

Subsequent applications include bounds for tri-colored sum-free sets
and limitations on the group-theoretic approach to matrix
multiplication~\cite{BCCGNSU17}. Naslund and Sawin~\cite{NS17} used
slice rank in the Erd\H{o}s--Szemer\'edi sunflower problem~\cite{ES78},
obtaining an exponential bound for three-sunflower-free families of
subsets of a fixed ground set.

We ask whether choosing a different tensor can extend these exponential
savings to four or more inputs. We fix the combinatorial relation but
allow the field and every nonzero tensor entry to vary. For both sunflowers and
arithmetic progressions, requiring the tensor to realize the entire
relation forces maximal slice rank. In the distinct-input variants
stated below, the slice rank remains a constant fraction of the ambient
size. Thus these classes of tensors cannot give an exponential
improvement over the trivial bound.

Tao and Sawin~\cite[Proposition~11]{TaoSawin16} previously proved a
related obstruction for progressions of length at least eight. Their
result assumes nonzero diagonal entries and zero entries outside the
progression relation, but permits zeros on nonconstant progressions.
Our finite-group result applies already at length four under the
stronger assumption that every progression receives a nonzero value.
These support hypotheses distinguish the two statements.

\subsection{Slice rank and support conditions}

Let $X$ be a finite set, $\F$ a field, and $k\ge2$. A
\emph{slice in direction $i$} is a function on $X^k$ of the form
$f(x_i)G(\xx_{-i})$, where
$\xx_{-i}$ denotes all inputs other than $x_i$. The \emph{slice rank}
$\sr(T)$ of $T:X^k\to\F$ is the minimum of $r_1+\cdots+r_k$ over all
representations
\[
 T(x_1,\ldots,x_k)
 =\sum_{i=1}^k\sum_{j=1}^{r_i}f_{i,j}(x_i)G_{i,j}(\xx_{-i}).
\]
Slicing along one direction gives $\sr(T)\le |X|$. Restricting the
inputs independently cannot increase slice rank, since a restricted
slice is still a slice or is zero. A tensor is \emph{diagonal} if it
vanishes unless $x_1=\cdots=x_k$. If all its diagonal entries are
nonzero, its slice rank is $|X|$~\cite{Tao16}.
Consequently, if $T$ restricts to such a tensor on $A^k$, then
$|A|\le\sr(T)$.

For a relation $R\subseteq X^k$, we distinguish two support conditions.
A tensor \emph{faithfully realizes $R$} if
\[
 T(x_1,\ldots,x_k)\ne0
 \quad\Longleftrightarrow\quad (x_1,\ldots,x_k)\in R
 \qquad\text{for every tuple in }X^k.
\]
It \emph{realizes $R$ on distinct inputs} if the same equivalence is
required only when $x_1,\ldots,x_k$ are pairwise distinct.
In the second definition, every entry on a tuple with a repetition is
unrestricted, including diagonal entries and entries outside $R$.
Both definitions allow arbitrary nonzero values and impose no symmetry
or product structure.

The second condition alone need not make a tensor diagonal on a
configuration-free set. For that application one may additionally
require nonzero diagonal entries and zero entries on all other tuples
with repetitions. Our lower bounds for distinct inputs hold without
these additional requirements, and therefore also cover this choice.

\subsection{Sunflowers}

A \emph{$k$-sunflower} is a collection of $k$ pairwise distinct sets
whose pairwise intersections are all equal. For indicator vectors in
$\{0,1\}^n$, equality of the pairwise intersections is equivalent to
requiring every coordinate pattern to belong to
\[
 \cS_k=\{0^k,e_1,\ldots,e_k,1^k\}\subseteq\{0,1\}^k,
\]
where $e_i$ is the $i$th unit vector. An element of the ground set is
absent from every set, belongs to exactly one set, or belongs to all
sets. We call the resulting relation on $(\{0,1\}^n)^k$ the
\emph{sunflower relation}. This relation allows repetitions.
Explicitly, faithful realization means
\begin{equation}\label{eq:faithful}
 T(x_1,\ldots,x_k)\ne0
 \quad\Longleftrightarrow\quad
 (x_{1,j},\ldots,x_{k,j})\in\cS_k\quad\text{for all }j\in[n].
\end{equation}

\paragraph{Three inputs.}

The construction of Naslund and Sawin~\cite{NS17} can be written over
$\mathbb F_3$ as
\[
 P_n(x,y,z)=\prod_{j=1}^n(2-x_j-y_j-z_j).
\]
A factor vanishes exactly when its three bits have weight two, so
$P_n$ faithfully realizes the three-sunflower relation. Its total
degree is at most $n$, and every monomial is multilinear. Each monomial
therefore has degree at most $\lfloor n/3\rfloor$ in at least one of
the three input groups $x,y,z$. Assign it to one such group. Collecting
terms with the same monomial in that group produces one slice. There
are $\sum_{j=0}^{\lfloor n/3\rfloor}\binom nj$ possible monomials in
each group, so
\begin{equation}\label{eq:three-bound}
 \sr(P_n)\le3\sum_{j=0}^{\lfloor n/3\rfloor}\binom nj
 \le\left(\frac{3}{2^{2/3}}\right)^{n+o(n)}.
\end{equation}
The base $3/2^{2/3}$ is strictly smaller than two.

To pass from a slice rank bound to a bound for ordinary sunflowers,
let $T$ faithfully realize the $k$-sunflower relation and let
$\mathcal A$ be a $k$-sunflower-free family. Restrict to a
\emph{uniform layer}, consisting of the sets in $\mathcal A$ of one fixed size.
If two sets in a tuple satisfying the sunflower relation are equal to
$C$, then $C$ is the common intersection and is contained in every set
in the tuple. If all sets have the same size, they must all equal $C$.
Since $\mathcal A$ contains no sunflower of pairwise distinct sets,
the restricted tensor is diagonal, with every diagonal entry nonzero.
Each uniform layer therefore has size at most $\sr(T)$. Summing over
the $n+1$ possible sizes gives
\begin{equation}\label{eq:family-bound}
 |\mathcal A|\le(n+1)\sr(T).
\end{equation}
Together with~\eqref{eq:three-bound}, this proves the exponential saving
for three sunflowers.

\paragraph{Our results.}
For four or more inputs, the same support condition forces maximal
slice rank.

\begin{theorem}[Full sunflower relation]\label{thm:main}
Let $\F$ be any field, $k\ge4$, and $n\ge0$. Every tensor
$T:(\{0,1\}^n)^k\to\F$ that faithfully realizes the sunflower
relation satisfies
\[
 \sr(T)=2^n.
\]
\end{theorem}

The conclusion is independent of the nonzero coefficients and the
coefficient field. The obstruction to exponential savings also holds
under the weaker requirement that only distinct inputs be prescribed.

\begin{corollary}[Distinct sunflower inputs]\label{cor:distinct-sunflower}
Let $\F$ be any field, $k\ge4$, and $n\ge k$. If
$T:(\{0,1\}^n)^k\to\F$ realizes the sunflower relation on distinct
inputs, then
\[
 \sr(T)\ge 2^{n-k}.
\]
\end{corollary}

For fixed $k$, the second bound is a constant fraction of the ambient
size. In particular, no family of these tensors can have slice rank
$O(c^n)$ for any $c<2$, even if all entries involving repetitions
are chosen freely. Both statements are proved in
Section~\ref{sec:sunflower-proof}.

\subsection{Arithmetic progressions}

Let $G$ be a finite abelian group, written additively. The
\emph{$k$-term arithmetic progression relation} consists of the tuples
\[
 (a,a+h,\ldots,a+(k-1)h),\qquad a,h\in G.
\]
We include constant progressions ($h=0$) and progressions with
repetitions caused by torsion. The coefficient field $\F$ is independent of the group $G$;
in particular, no assumption on its characteristic is made.

\begin{theorem}[Full progression relation]\label{thm:ap}
Let $G$ be a finite abelian group, $\F$ any field, and $k\ge4$.
If $T:G^k\to\F$ faithfully realizes the $k$-term arithmetic progression
relation, then
\[
 \sr(T)=|G|.
\]
\end{theorem}

In finite vector spaces, the same obstruction persists when only
pairwise distinct inputs are prescribed.

\begin{corollary}[Distinct progression inputs]\label{cor:distinct-ap}
Let $k\ge4$, let $p\ge k$ be prime, and let $n\ge1$.
For any coefficient field $\F$, if
$T:(\mathbb F_p^n)^k\to\F$ realizes the $k$-term arithmetic progression
relation on distinct inputs, then
\[
 \sr(T)\ge p^{n-1}.
\]
\end{corollary}

We prove Theorem~\ref{thm:ap} and Corollary~\ref{cor:distinct-ap}
in Section~\ref{sec:ap}.
We note that for finite subsets of the integers, the corresponding maximal slice
rank bound follows from Tao and Sawin~\cite[Proposition~4]{TaoSawin16},
already for $k\ge3$.

\paragraph{Partition rank.}
Partition rank, introduced by Naslund~\cite{Nas20}, is the minimum
number of terms $f((x_i)_{i\in I})g((x_i)_{i\notin I})$ needed to
express a tensor $T$, where $\varnothing\ne I\subsetneq[k]$ may vary
between terms. We denote it by $\pr(T)$. Every slice is such a term,
so $\pr(T)\le\sr(T)$. Partition rank does not increase under
restriction, and a diagonal tensor with $N$ nonzero entries has
partition rank $N$~\cite[Lemma~11]{Nas20}. Thus the same restriction
argument can bound configuration-free sets using partition rank in
place of slice rank, potentially giving stronger bounds. In further
work with ChatGPT, we obtained a lower bound of $1.96^n$ for
$4$-sunflowers, while the trivial upper bound is $2^n$. We omit the
proof, as it is technical and offers little additional intuition.
It remains open to determine the asymptotic minimum partition rank
of tensors faithfully realizing the sunflower and arithmetic
progression relations for $k\ge4$. Already for $4$-sunflowers, must
this minimum be $2^{n-o(n)}$, or can it be $O(c^n)$ for some $c<2$?

\paragraph{AI methodology.}

We used ChatGPT-6 Astra to explore proof strategies, develop and check
candidate arguments, and iteratively refine the writing. During the exploratory
stage, multiple AI agents investigated different approaches and
cross-checked the sunflower lower-bound argument. The proofs and their
presentation were then refined through an interactive exchange of
mathematical questions, proposed reductions, and requests to clarify
individual steps. Particular emphasis was placed on expressing the
sunflower proof directly in terms of slice decompositions and making
the support conditions explicit. 

\section{Sunflowers}\label{sec:sunflower-proof}

We first prove Theorem~\ref{thm:main}, then deduce
Corollary~\ref{cor:distinct-sunflower} by distinguishing the inputs
with fixed prefixes.

The proof of the theorem is an induction on $n$. Splitting the inputs
according to their first coordinate divides $T$ into blocks. Each
nonzero block has slice rank $2^{n-1}$ by induction. The main
step is to show that the original tensor needs at least twice as many
slices. A single slice can contribute to several blocks, so we cannot
simply add their ranks. Instead, we use linear maps that recover
selected blocks while removing many slices from a given decomposition,
and combine the resulting inequalities.

\subsection{The first-coordinate blocks}\label{sec:blocks}

To describe the induction step, fix a field $\F$, integers $k\ge4$ and
$n\ge1$, and a tensor $T$ satisfying~\eqref{eq:faithful}.
Let $X=\{0,1\}^n$ and write
\[
  X_0=\{(0,u):u\in\{0,1\}^{n-1}\},\qquad
  X_1=\{(1,u):u\in\{0,1\}^{n-1}\}.
\]
We identify each $X_\varepsilon$ with $\{0,1\}^{n-1}$ by deleting
its first coordinate. The partition $\{0,1\}^n=X_0\sqcup X_1$ divides
$T$ into $2^k$ subtensors, one for each $\omega\in\{0,1\}^k$:
\begin{equation}\label{eq:blocks}
  T_\omega(u_1,\ldots,u_k)
  =T\bigl((\omega_1,u_1),\ldots,(\omega_k,u_k)\bigr),
  \qquad u_i\in\{0,1\}^{n-1}.
\end{equation}
Their support is completely determined by~\eqref{eq:faithful}.
If $\omega\notin\cS_k$, the first coordinate is forbidden, and
$T_\omega$ is identically zero. Equivalently, all blocks indexed by
patterns of weight $2,\ldots,k-1$ vanish.

If $\omega\in\cS_k$, the first coordinate already satisfies the
sunflower condition. Hence
\begin{equation}\label{eq:block-support}
  T_\omega(u_1,\ldots,u_k)\ne0
  \quad\Longleftrightarrow\quad
  (u_{1,j},\ldots,u_{k,j})\in\cS_k
  \quad\text{for all }j\in[n-1].
\end{equation}
Thus each of the $k+2$ allowed blocks faithfully realizes the same
relation on $n-1$ coordinates. In particular,
\[
  T_{e_1},\ldots,T_{e_k},T_{1^k}
\]
satisfy every hypothesis needed for the induction. Their nonzero
values can differ from block to block; the inductive assertion covers
all such choices. The block $T_{0^k}$ also has this property, but its
slice rank is not needed in the proof.

\subsection{Organizing and cancelling the slices}\label{sec:cancellation}

Our maps will retain one half of the inputs and cancel selected
factors by subtracting linear combinations of values on the other
half. The following lemma supplies the required maps.
For a finite index set $Y$,
write $\F^Y$ for the space of functions $Y\to\F$. If $Z\subseteq Y$,
we identify $\F^Z$ with the coordinate subspace of $\F^Y$ consisting
of functions supported on $Z$.

\begin{lemma}[Cancelling a subspace]\label{lem:cancel}
Let $[N]=A\sqcup B$, let $V\subseteq\F^N$ be a subspace, and put
$V_A=V\cap\F^A$. For any complement $U$ with $V=V_A\oplus U$, there
is a linear map $C:\F^N\to\F^A$ such that
\[
  C(a)=a\quad\text{for every }a\in\F^A,
  \qquad
  C(u)=0\quad\text{for every }u\in U.
\]
\end{lemma}

\begin{proof}
Since $U\subseteq V$, we have $U\cap\F^A=U\cap V_A=\{0\}$.
Thus $C(a+u)=a$ defines a linear map on $\F^A\oplus U$.
Extend this map linearly to $\F^N$.
\end{proof}

We apply the lemma to the one-variable factors in a slice decomposition.
Choose a minimum decomposition
\begin{equation}\label{eq:decomposition}
  T(x_1,\ldots,x_k)
  =
  \sum_{i=1}^k\sum_{j=1}^{r_i}
  f_{i,j}(x_i)G_{i,j}(\xx_{-i}),
  \qquad
  r:=\sr(T)=\sum_{i=1}^k r_i.
\end{equation}
Within each direction the factors $f_{i,j}$ are linearly independent:
otherwise we could combine terms and reduce the number of slices.
Changing their basis changes the complementary factors $G_{i,j}$ but
preserves the decomposition and its number of terms. Let
\[
  V_i=\operatorname{span}_{\F}\{f_{i,1},\ldots,f_{i,r_i}\}
  \subseteq\F^X.
\]
In $V_i$, first choose bases of $V_i\cap\F^{X_0}$ and
$V_i\cap\F^{X_1}$. These two subspaces have zero intersection.
Extend the union of their bases to a basis of $V_i$. Denote the
dimensions of the two supported subspaces and their complement by
\[
  a_i=\dim(V_i\cap\F^{X_0}),\qquad
  b_i=\dim(V_i\cap\F^{X_1}),\qquad
  m_i=r_i-a_i-b_i.
\]
We call the last $m_i$ basis factors \emph{mixed}. Each has a nonzero
restriction to both halves, so restrictions to either half retains it.
The maps below will cancel these mixed factors. Let
\begin{equation}\label{eq:totals}
  A=\sum_i a_i,\qquad B=\sum_i b_i,\qquad M=\sum_i m_i,
  \qquad r=A+B+M.
\end{equation}

Fix $i\in[k]$ and $\varepsilon\in\{0,1\}$. Apply Lemma~\ref{lem:cancel}
with coordinates indexed by $X$, with $V=V_i$, and with the
partition $X=X_\varepsilon\sqcup X_{1-\varepsilon}$. Take $U$ to be
the span of the mixed factors and the factors supported on
$X_{1-\varepsilon}$. Our choice of basis gives
\[
  V_i=(V_i\cap\F^{X_\varepsilon})\oplus U.
\]
The resulting map $C_i^\varepsilon:\F^X\to\F^{X_\varepsilon}$
fixes the basis factors supported on $X_\varepsilon$ and kills all
the other basis factors, since they belong to $U$.

Because $C_i^\varepsilon$ fixes \emph{every} function supported on
$X_\varepsilon$, it is the restriction to $X_\varepsilon$ plus a linear
correction depending only on the opposite half. In coordinates, there
are coefficients $\lambda_{t,u}\in\F$ such that
\begin{equation}\label{eq:map-form}
  (C_i^\varepsilon f)(t)
  =f(t)-\sum_{u\in X_{1-\varepsilon}}\lambda_{t,u}f(u),
  \qquad t\in X_\varepsilon.
\end{equation}

Apply $C_i^\varepsilon$ to a tensor by acting on input $i$ and holding
all other inputs fixed. For a direction-$i$ slice, this replaces its
one-variable factor $f$ by $C_i^\varepsilon f$. If $f\in U$, the
resulting slice is zero. Hence at most $a_i$
slices in that direction remain when $\varepsilon=0$, and at most
$b_i$ when $\varepsilon=1$. For a slice in direction $j\ne i$, its
one-variable factor is unchanged and only its complementary factor is
modified. Thus the number of slices in any other direction cannot
increase. The operations in distinct directions commute, so these
conclusions remain valid when several of them are applied successively.

These operations need not preserve a tensor's values. In the proof
below, we arrange that every correction term evaluates $T$ on a
forbidden block, where it is zero. The maps then recover the desired
block while still cancelling the specified slices.

\subsection{The inductive proof}\label{sec:proof}

\begin{proof}[Proof of Theorem~\ref{thm:main}]
For $n=0$, the tensor is a nonzero scalar, so its slice rank is one.
Suppose $n\ge1$ and the result holds for $n-1$. Put
$s=2^{n-1}$, and organize a minimum slice decomposition of $T$ as above.
We will prove $r\ge2s$. The reverse inequality follows by slicing along
any one direction, which has $2s$ possible inputs.

Among the $2^k$ blocks defined in Section~\ref{sec:blocks}, set
\[
  H=T_{1^k},\qquad S_i=T_{e_i}\quad(1\le i\le k).
\]
Each satisfies the hypotheses of Theorem~\ref{thm:main} on $n-1$
coordinates, by~\eqref{eq:block-support}. Their nonzero coefficients
may be arbitrary and unrelated across blocks. By induction,
\begin{equation}\label{eq:child-ranks}
  \sr(H)=\sr(S_i)=s
  \qquad (1\le i\le k).
\end{equation}

\Needspace{7\baselineskip}
\medskip
\noindent\emph{The all-one block.}
Choose distinct directions $i,j \in [k]$. In these two directions, simply
restrict the inputs to $X_1$. In every other direction $\ell$, apply
$C_\ell^1$ constructed in Section~\ref{sec:cancellation}, retaining $X_1$ and
using evaluations on $X_0$.
The two ordinary restrictions preserve two first-coordinate ones.
Every correction inserts at least one zero, so its pattern will have
between two and $k-1$ ones and hence belong to a zero block. We now
write this out explicitly.

Let $I=[k]\setminus\{i,j\}$, and denote the resulting tensor on
$X_1^k$ by $\widetilde H$. For $\ell\in I$, write
$\lambda^{(\ell)}_{t,u}$ for the coefficients of $C_\ell^1$
in~\eqref{eq:map-form}. Fix $\mathbf t=(t_1,\ldots,t_k)\in X_1^k$.
For $J\subseteq I$ and $\mathbf u=(u_\ell)_{\ell\in J}\in X_0^J$,
let $\mathbf t^{J,\mathbf u}$ be the tuple whose $\ell$th input is
$u_\ell$ if $\ell\in J$, and $t_\ell$ otherwise. Distributing the
restriction and correction terms in~\eqref{eq:map-form} gives
\begin{equation}\label{eq:all-one-expansion}
  \widetilde H(\mathbf t)
  =\sum_{J\subseteq I}(-1)^{|J|}
    \sum_{\mathbf u\in X_0^J}
    \left(\prod_{\ell\in J}\lambda^{(\ell)}_{t_\ell,u_\ell}\right)
    T(\mathbf t^{J,\mathbf u}).
\end{equation}
Here $J$ records precisely the directions in which the correction
term is chosen. The term with $J=\varnothing$ is $T(\mathbf t)$.
If $J\ne\varnothing$, the first-coordinate pattern of
$\mathbf t^{J,\mathbf u}$ has zeros exactly in the positions in $J$.
Its number of ones is therefore $k-|J|$. Since
$1\le |J|\le k-2$, this number lies between $2$ and $k-1$.
Thus the pattern is forbidden and
$T(\mathbf t^{J,\mathbf u})=0$ for every such $J$ and $\mathbf u$.
It follows that $\widetilde H(\mathbf t)=T(\mathbf t)$ for every
$\mathbf t\in X_1^k$, so the resulting tensor is exactly the
restriction $H$.

We can now count the slices that remain in its decomposition.
In every direction $\ell\notin\{i,j\}$, only the $b_\ell$
factors supported on $X_1$ survive. In directions $i,j$, the
restriction to $X_1$ leaves at most $b_i+m_i$ or $b_j+m_j$ factors.
Consequently,
\[
  s=\sr(H)\le B+m_i+m_j.
\]
Each $m_\ell$ occurs in $k-1$ of the $\binom{k}{2}$ inequalities.
Averaging over the unordered pairs $\{i,j\}$ therefore gives
\begin{equation}\label{eq:all-one}
  s\le B+\frac{2M}{k}.
\end{equation}

\Needspace{7\baselineskip}
\medskip
\noindent\emph{The singleton blocks.}
Fix $i$, and choose $j\ne i$. In direction $i$, simply restrict to
$X_1$; in direction $j$, simply restrict to $X_0$.
In every other direction $\ell$, apply $C_\ell^0$, retaining
$X_0$ and using evaluations on $X_1$.
The ordinary restrictions now preserve one first-coordinate one and
one zero. Every correction adds another one, again forcing a forbidden
pattern.

Expand as in~\eqref{eq:all-one-expansion}, with
$I=[k]\setminus\{i,j\}$, now using the coefficients of $C_\ell^0$.
The target inputs satisfy $t_i\in X_1$ and $t_\ell\in X_0$ for
$\ell\ne i$, and the substituted inputs satisfy $u_\ell\in X_1$.
The term indexed by $J=\varnothing$ is the unmodified value of
$S_i$. In every term indexed by a nonempty $J\subseteq I$, the
first-coordinate pattern has ones exactly in $\{i\}\cup J$.
It therefore has $1+|J|$ ones, again between $2$ and $k-1$.
The corresponding value of $T$ is zero, so the resulting tensor is
exactly $S_i$.

In direction $i$, at most $b_i+m_i$ factors survive. In direction
$j$, at most $a_j+m_j$ survive. In every other direction $\ell$,
only the $a_\ell$ factors supported on $X_0$ remain. Therefore
\begin{equation}\label{eq:singleton}
  s=\sr(S_i)
  \le b_i+\sum_{\ell\ne i}a_\ell+m_i+m_j.
\end{equation}
For each $i$, take $j=i+1$, with indices read cyclically modulo $k$,
and sum. Thus each direction plays the role of $i$ once and of $j$
once. Each $a_\ell$ is counted $k-1$ times, each $b_\ell$ once,
and each $m_\ell$ twice. Thus
\begin{equation}\label{eq:singleton-sum}
  ks\le B+(k-1)A+2M.
\end{equation}

\Needspace{6\baselineskip}
\medskip
\noindent\emph{Combining the two counts.}
Multiply \eqref{eq:all-one} by $k-2$ and add
\eqref{eq:singleton-sum}, so that $A$ and $B$ have the same
coefficient. This yields
\[
  2(k-1)s
  \le (k-1)(A+B)+\left(4-\frac4k\right)M.
\]
Dividing by $k-1$, we obtain
\[
  2s\le A+B+\frac4kM
  \le A+B+M=r,
\]
where the last inequality uses $k\ge4$. This proves the induction.
\end{proof}

All averages and divisions in the proof concern ordinary integer
dimensions and slice counts, not elements of $\F$. The proof therefore
works in every characteristic.

\begin{remark}[The unused block]
The proof never uses $T_{0^k}$. In fact, replacing that entire block by
an arbitrary tensor leaves the conclusion unchanged, provided all other
blocks remain as above. Indeed, the inductive lower bounds only use
$T_{1^k}$, the $T_{e_i}$, and the all-zero sub-tensors of $T$. Every correction term has between
two and $k-1$ ones in its first-coordinate pattern, so none reaches
$T_{0^k}$.
\end{remark}

\subsection{Distinct inputs}

\begin{proof}[Proof of Corollary~\ref{cor:distinct-sunflower}]
Use the first $k$ coordinates to distinguish the inputs. For
$u_1,\ldots,u_k\in\{0,1\}^{n-k}$, define the subtensor
\[
 S(u_1,\ldots,u_k)
 =T\bigl((e_1,u_1),\ldots,(e_k,u_k)\bigr),
 \qquad e_i\in\{0,1\}^k.
\]
The $k$ prefixed inputs are pairwise distinct for every choice of the
$u_i$, even when some or all of the $u_i$ coincide. In prefix coordinate
$j$, their pattern is $e_j$, which is allowed. Therefore
\[
 S(u_1,\ldots,u_k)\ne0
 \quad\Longleftrightarrow\quad
 (u_1,\ldots,u_k)\text{ satisfies the sunflower relation}.
\]
Thus $S$ faithfully realizes the full relation on $n-k$ coordinates.
By Theorem~\ref{thm:main} and monotonicity under restriction,
$2^{n-k}=\sr(S)\le\sr(T)$.
\end{proof}

\section{Arithmetic progressions}\label{sec:ap}

We first prove Theorem~\ref{thm:ap}. With four inputs, we view the
tensor as a matrix of rank $|G|^2$, while each slice gives a matrix of
rank at most $|G|$. With more inputs, we sum against functions that annihilate all
slices in the extra directions while preserving enough nonzero matrix
entries. We then apply the same matrix argument and finally deduce
Corollary~\ref{cor:distinct-ap}.

\subsection{Four inputs: a matrix of full rank}

\begin{proof}[Proof of Theorem~\ref{thm:ap} for $k=4$]
Put $N=|G|$. For a four-tensor $S$, group the first two inputs as a row
index and the last two as a column index. The resulting matrix, called
the $12\mid34$ \emph{flattening}, is
\[
 M_S\bigl((x_1,x_2),(x_3,x_4)\bigr)=S(x_1,x_2,x_3,x_4).
\]
This is an $N^2\times N^2$ matrix. If $S$ faithfully realizes the
four-term progression relation, then row $(x,y)$ has exactly one
nonzero entry, in column
\[
 \phi(x,y)=(2y-x,\,3y-2x).
\]
The map $\phi:G^2\to G^2$ is bijective: its inverse sends $(z,w)$ to
$(3z-2w,\,2z-w)$. These formulas use only integer multiples in $G$,
so they hold in every finite abelian group. Thus $M_S$ is a weighted
permutation matrix and has rank $N^2$ over every coefficient field.

In contrast, the flattening of each slice has matrix rank at most
$N$. For a direction-one slice, for example,
\[
 f(x_1)H(x_2,x_3,x_4)
 =\sum_{b\in G}
   \bigl(f(x_1)\mathbf1_{x_2=b}\bigr)H(b,x_3,x_4)
\]
is a sum of $N$ matrices of rank at most one. The same argument applies
to direction two. Directions three and four follow by transposing the
matrix. By subadditivity of matrix rank, a decomposition of $S$ into
$R$ slices therefore implies
\begin{equation}\label{eq:flatten-upper}
 \rank(M_S)\le NR.
\end{equation}
For a faithful progression tensor, this gives $N^2\le N\sr(S)$,
proving the theorem when $k=4$.
\end{proof}

\subsection{A linear functional with few zero coordinates}

For functions $f,h:X\to\F$, use the bilinear pairing
$\langle f,h\rangle=\sum_{x\in X}f(x)h(x)$.
Summing against $h$ removes a slice with factor $f$ whenever
$\langle f,h\rangle=0$. We need such a function $h$ that vanishes at
few points: these zeros will be the only reason a progression entry is
lost. The following lemma gives both properties, with the size of the
zero set controlled by the dimension of the space to be annihilated.

\begin{lemma}\label{lem:annihilator}
Let $X$ have size $N$, and let $V\subseteq\F^X$ have dimension $r$.
There is a function $h:X\to\F$ such that
\[
 \langle f,h\rangle=0\quad(f\in V),
 \qquad
 |\{x\in X:h(x)=0\}|\le r.
\]
\end{lemma}

\begin{proof}
Let $U=V^\perp$ where $\dim U=m:=N-r$.
Write a basis of $U$ as the rows of a matrix and put it in reduced
row-echelon form. Denote the
resulting basis by $h_1,\ldots,h_m$ and its pivot coordinates by
$x_1,\ldots,x_m\in X$. Thus
\[
 h_i(x_j)=\delta_{ij}\qquad(1\le i,j\le m).
\]
The function $h=h_1+\cdots+h_m$ belongs to $V^\perp$ and satisfies
$h(x_j)=1$ at each of the $m=N-r$ pivot coordinates. Hence it has
at most $r$ zero coordinates.
\end{proof}

\subsection{More than four inputs}

\begin{proof}[Proof of Theorem~\ref{thm:ap} for $k>4$]
Write $N=|G|$ and choose a minimum slice decomposition
\[
 T(x_1,\ldots,x_k)
 =\sum_{i=1}^k\sum_{j=1}^{r_i}f_{i,j}(x_i)G_{i,j}(\xx_{-i}),
 \qquad R:=\sr(T)=\sum_{i=1}^k r_i.
\]
As in the sunflower proof, the factors in each direction are linearly
independent. Put
\[
 V_i=\operatorname{span}_{\F}\{f_{i,1},\ldots,f_{i,r_i}\},
 \qquad \dim V_i=r_i.
\]
For every $i=5,\ldots,k$, apply Lemma~\ref{lem:annihilator} to choose
a function $h_i:G\to\F$ orthogonal to every element of $V_i$, with
at most $r_i$ zero coordinates. Define
\begin{equation}\label{eq:contract-ap}
 S(x_1,x_2,x_3,x_4)
 =\sum_{x_5,\ldots,x_k\in G}
   T(x_1,\ldots,x_k)\prod_{i=5}^k h_i(x_i).
\end{equation}

Consider this operation on the chosen slice decomposition. A slice in
direction $i\ge5$ contributes zero, because summing in $x_i$ gives the
factor $\langle f_{i,j},h_i\rangle=0$. A slice in direction $i\le4$
remains a slice in that direction, with its complementary factor
summed as in~\eqref{eq:contract-ap}. Therefore $S$ has a decomposition
with at most $r_1+\cdots+r_4$ slices. By~\eqref{eq:flatten-upper},
\begin{equation}\label{eq:contract-upper}
 \rank(M_S)\le N\sum_{i=1}^4r_i.
\end{equation}

We now count the nonzero entries of $M_S$ from its support. A
four-term progression $(a,a+h,a+2h,a+3h)$ has exactly one completion
to a $k$-term progression, since its first two entries determine both
$a$ and $h$. Hence the contraction sum has at most one nonzero term
for each fixed $(x_1,x_2,x_3,x_4)$; in particular, there is no
cancellation between different completions. Consequently,
\begin{equation}\label{eq:contract-values}
 S(a,a+h,a+2h,a+3h)
 =T(a,a+h,\ldots,a+(k-1)h)
          \prod_{i=5}^k h_i\bigl(a+(i-1)h\bigr),
\end{equation}
and $S$ vanishes outside the four-term progression relation. The
first factor on the right is nonzero by faithful realization.
Thus $M_S$ is obtained from a weighted permutation matrix by possibly
replacing some of its nonzero entries by zero.

Let $Z_i=\{x\in G:h_i(x)=0\}$, so $|Z_i|\le r_i$. For each
fixed $i$ and each $z\in Z_i$, exactly $N$ pairs $(a,h)\in G^2$
satisfy $a+(i-1)h=z$: choose $h$ arbitrarily and then $a$ is determined.
A union bound shows that at most $N\sum_{i=5}^k r_i$ of the $N^2$
entries in~\eqref{eq:contract-values} vanish. The surviving entries
lie in distinct rows and distinct columns, so
\begin{equation}\label{eq:contract-lower}
 \rank(M_S)\ge N^2-N\sum_{i=5}^k r_i.
\end{equation}
Thus removing the $r_i$ slices in an extra direction costs at most
$Nr_i$ in the matrix rank lower bound, exactly the contribution allowed
for $r_i$ slices in the first four directions by~\eqref{eq:flatten-upper}.
Combining~\eqref{eq:contract-upper} and~\eqref{eq:contract-lower} gives
\[
 N^2\le N\sum_{i=1}^k r_i=NR.
\]
Hence $R\ge N$. The upper bound $\sr(T)\le N$ completes the proof.
\end{proof}

\subsection{Distinct inputs}

We use the first coordinates $0,1,\ldots,k-1$ to distinguish the
inputs while preserving the progression relation.

\begin{proof}[Proof of Corollary~\ref{cor:distinct-ap}]
For $u_i\in\mathbb F_p^{n-1}$, restrict to
\[
 S(u_1,\ldots,u_k)
 =T\bigl((0,u_1),(1,u_2),\ldots,(k-1,u_k)\bigr).
\]
The first coordinates are pairwise distinct because $p\ge k$.
Thus every tuple queried in $T$ has distinct inputs, even when some
of the $u_i$ coincide. Its first coordinates already form a
progression, so the full tuple is a progression if and only if
$(u_1,\ldots,u_k)$ is a progression in $\mathbb F_p^{n-1}$.
Consequently, $S$ faithfully realizes the full progression relation
on the additive group of $\mathbb F_p^{n-1}$.
Theorem~\ref{thm:ap} and monotonicity under restriction give
\[
 p^{n-1}=\sr(S)\le\sr(T).
\]
\end{proof}

\bibliographystyle{plainnat}
\bibliography{slice_rank_limitations}

\end{document}